\documentclass[10pt,leqno]{article}

\usepackage{amsmath,amsthm,amsfonts,latexsym,amscd,amssymb,xcolor}
\numberwithin{equation}{section}
\input xypic
\def\qed{{\hbadness=10000\hfill\ \vbox{\hrule height.09ex
			\hbox{\vrule width.09ex height1.55ex depth.2ex \kern1.8ex
				\vrule width.09ex height1.55ex depth.2ex}\hrule height.09ex}\break
		\bigskip}}
\newtheorem{theorem}{Theorem}[section]
\newtheorem{lemma}{Lemma}[section]

\theoremstyle{definition}
\newtheorem{definition}[theorem]{Definition}
\newtheorem{example}[theorem]{Example}

\theoremstyle{remark}

\begin{document}
	\date{}
	
	\linespread{1}\title{\textbf{\LARGE Some Novel Results on $(\alpha,\beta)$-Ricci-Yamabe Soliton and its Spacetime}} 
	\author{{$^{1}$ Pankaj Pandey and $^{2}$ Kamakshi Sharma}\\
		\normalsize{$^{1, 2}$Department of Mathematics (SCEPS)}\\
		\normalsize{Lovely Professional University, Phagwara, Punjab-144411, India}\\
		\normalsize{$^{1}$ Email: pankaj.anvarat@gmail.com}\\
		\normalsize{$^{2}$ Email: kamakshisharma82@gmail.com}}
	\maketitle
	
\noindent\textbf{Abstract:} This article aims to investigate the characteristics of $(\alpha,\beta)$-Ricci-Yamabe Soliton (briefly: $(\alpha,\beta)-(RYS)_n$) and its spacetime. The inclusion of killing vector field and the Lorentzian metrics make the Ricci-Yamabe soliton richer and interesting.  We study the cosmological and dust fluid model on $(RYS)_4$ equipped with Lorentzian para Sasakian $(LPS)_4$ spacetime. The cases of $\eta$-parallel Ricci tensor and the Poisson structure have been studied on $(RYS)_n$ equipped with $(LPS)_n$ manifold. Gradient $(RYS)_n$ equipped with $(LPS)_n$ manifold also reveal. Finally, we establish an example of four-dimensional LP-Sasakian manifold $(LPS)_4$ that satisfy  $(\alpha,\beta)-(RYS)_4$ and some results.



\section{Introduction} 
\indent Hamilton \cite{rshamilton82} revealed the concept of Ricci flow (resp. Yamabe flow) in the last quarter of twenties in order to discuss new striking results in Riemannian geometry. The idea of Ricci soliton recognized as a generalization of an Einstein metric and governing as the solution of partial differential equation representing Ricci flow which is isomorphically equivalent to heat equation. In case of Ricci flow, the Riemannian metric is proportional to a (0, 2) type tensor ($\frac{1}{2}Lg+Ric$). In this case the proportionality constant is called the soliton constant. The Ricci flow in terms of soliton constant $\lambda$ (say) is governed by the equation \cite{rshamilton82}
\begin{equation}\label{eqn. 1.1a}
	-\frac{1}{2}(L_{v}g)(\zeta_{1},\zeta_{2})=S(\zeta_{1},\zeta_{2})+\lambda g(\zeta_{1},\zeta_{2}),
\end{equation}
which is equivalent to
\begin{equation}\label{eqn. 1.1b}
	\frac{\partial g(t)}{\partial t} = -2 S(t), ~~~~~~g_{0} = g(0).
\end{equation}
The Yamabe flow is discribed by an evolving equation \cite{dcy18, rshamilton82}
\begin{equation}\label{eqn. 1.1c}
	\frac{\partial g(t)}{\partial t}= -r(t)g(t), ~~~~~~g_{0} = g(0),
\end{equation}
equivalently
\begin{equation}\label{eqn. 1.1d}
	\frac{1}{2}(L_{v}g)(\zeta_{1},\zeta_{2})=(r-\lambda)g(\zeta_{1},\zeta_{2}).
\end{equation}
In (\ref{eqn. 1.1a}) and (\ref{eqn. 1.1d}), the Lie derivative $L$ is taken along the complete vector field $v$ and is denoted by $L_vg$. The Ricci denotes $S$ and the $r$ is scalar tensor. The vector fields $\zeta_{1},\zeta_{2}$ belong to the set of an algebra of tangent vectors denoted by $\chi(M)$. The nature of the soliton can be expressed in terms of soliton constant $\lambda$ and is said to explore for $\lambda$ positive. We say the case is of compacting soliton or is of constant soliton, if $\lambda<0$ or $\lambda=0$. In general the Ricci and the Yamabe solitons are distinct for higher dimension but coincide in case of dimension 2. The reason behind this distinction is that the Yamabe soliton preserves the metric conformality in nature where as Ricci soliton denied it.\\
 \indent Guler and Crasmareanu \cite{sgmc19} $(2019)$ investigated another geometric flow restricted by Ricci-Yamabe map. Authors pronounced with a special call: $(\alpha, \beta)$ type Ricci-Yamabe flow (briefly denoted $(\alpha, \beta)- (RYF)_n$). Particularly, it is noted that the $(\alpha, \beta)- (RYF)_n$ is nothing but the $\alpha$-Ricci soliton for $\beta=0$ and it turns into $\beta$-Yamabe soliton, in case of vanishing $\alpha$. The above said flow, in the Riemannian manifold, is defined by 
\begin{equation}\label{eqn. 1.1}
	\frac{\partial}{\partial t}g(t) = \beta r(t)g(t) - 2\alpha S(t), ~~~~~~g_{0} = g(0),
\end{equation}
where $S$ being Ricci tensor, $r$: the scalar curvature, and $\lambda, \alpha, \beta \in R$. As a restriction to solve the $(\alpha, \beta)- (RYF)_n$, the solution of nonlinear partial differential equation (\ref{eqn. 1.1}) notified as $(\alpha, \beta)- (RYS)_n$. \\
\indent If we denote $(M^n,g)$ by an $n-$dimensional $(\alpha, \beta)- (RYS)_n$, on $(M^n,g)$ having structure $(g,V,\lambda,\alpha,\beta)$ that gratify the equation
\begin{equation}\label{eqn. 1.2}
	(L_{V}g)(\zeta_{1},\zeta_{2}) = -2\alpha S(\zeta_{1},\zeta_{2}) - (2\lambda - \beta r)g(\zeta_{1},\zeta_{2}).
\end{equation}
The aforementioned soliton is known as the $(\alpha, \beta)$ type gradient Ricci-Yamabe soliton (briefly $(\alpha, \beta)- (GRYS)_n$) if $V$ becomes a gradient of $f$ (if exist such a smooth function) and (\ref{eqn. 1.2}) reveals
\begin{equation}\label{eqn. 1.3}
	\nabla^{2}f+\alpha S=(\lambda-\frac{1}{2}\beta r)g,
\end{equation}
where $	\nabla^{2}f$ is defined as Hessian of the smooth function $f$ and denoted in general by $Hess(f)$.
The nature of the $(\alpha, \beta)- (RYS)_n$) is quite similar and depends on the soliton constant $\lambda$. Accordingly, the respective soliton is called steady for $\lambda=0$ otherwise natured as shrinking for being $\lambda$ negative. The positiveness of $\lambda$ explored as expanding.\\
An $(\alpha, \beta)- (RYS)_n$) turns into almost $(\alpha, \beta)- (RYS)_n$ (briefly $(\alpha, \beta)- (ARYS)_n$) for ($\lambda, \alpha, \beta$) being smooth function. The $(\alpha, \beta)- (RYS)_n$ simplifying as Ricci soliton in special case $( \alpha=1, \beta=0)$ and it turns into Yamabe soliton in other special case $(\alpha=0, \beta=1)$. Also, for $(\alpha=1, \beta=-1)$ and $(\alpha=1, \beta=-2\rho)$, it pointed into Einstein and $\rho$-Einstein soliton. Some times, we call it proper for the values  $\alpha \not= 0, 1$. It is important to keep in mind that the Ricci Yamabe soliton corresponds to a specific class of generalized Ricci solitons [16] if scalar curvature $r$ is constant in (\ref{eqn. 1.2}). Therefore, studying Ricci Yamabe soliton with non-constant scalar curvature rather than the constant scalar curvature is more appropriate and natural.\\
\indent The Lorentzian manifold,  which is among the most significant sub classes of pseudo-Riemannian manifold, has a significant impact on the advancement of the theory of relativity and cosmology. Ahsan and Ali \cite{mza14} looked into the symmetries of soliton spacetime. The geometrical characteristics discussed by Blaga \cite{amb20} in fluid spacetime under assumptions of Einstein/Ricci soliton. In \cite{amb20} author examined the geometrical importance of $(PFST)_n$ conformal Ricci solitons. Chaturvedi et.al. studied the K$\ddot{a}$hler spacetime manifold under consideration of Bochner flatness. They determined that the energy-momentum tensor of a perfect fluid Lorentzian K$\ddot{a}$hler spacetime manifold exhibits hybrid characteristics. They also analyzed the behaviour of dust fluid Lorentzian K$\ddot{a}$hler spacetime manifold where the Bochner curvature vanishes (see also authored by Pokhriyal and Chaubey \cite{gsk23}).\\
\indent The LP-Kenmotsu manifold was investigated by Haseeb and Prasad in \cite{hpr21}. Haseeb et.al. \cite{hbca22} proposed the study on $(\alpha, \beta)- (RYS)_n$ and $(\alpha, \beta)- (GRYS)_n$ conditionally related to the $\zeta$-conformally flat $m-$dimensional LP Kenmotsu manifold. They proved that the scalar tensor of $(LPK)_{m}$ manifold admitting Ricci Yamabe soliton satisfies the Poisson equation. Sardar and Sarkar \cite{asas21} described Kenmotsu $3$-manifold equipped with $(\alpha, \beta)- (RYS)_n$ and $(\alpha, \beta)- (GRYS)_n$ metric that satisfy $\zeta$-parallel Ricci tensor. Recently, Pal and Chaudhary \cite{bprsc23} revealed the idea of Poisson fluid flow to investigate almost soliton like Ricci soliton (briefly: $(ASLRS)_n$) by exploring the statistical structure.\\
De and De \cite{dedekmj284} (2020) studied almost Ricci soliton and gradient almost Ricci soliton on para-Sasakian manifolds. Kundu \cite{kundukmj292} (2021) explored para-Kenmotsu metric as an $\eta$-Ricci soliton. In \cite{ddekmj29} (2021), De discussed Sasakian 3-manifolds admitting a gradient Ricci-Yamabe soliton. Sarkar and Bhakta \cite{sarkarbhakta} obtained some interesting results on certain solitons on generalized $(\kappa, \mu)$ contact metric manifolds. D. Kar and P. Majhi \cite{dkarmajhikmj} (2019) concentrated on Beta-almost Ricci solitons on almost Co-Kahler manifolds. Vishwas, Das, Baishya and Bakshi \cite{avadkkb20} studied $\eta$-Ricci solitons on Kenmotsu manifolds admitting general connection. T. Mandal \cite{tmandal22} (2022) given certain results on three-dimensional $f$-Kenmotsu manifolds with conformal Ricci solitons. Lee, Kim and Choi \cite{sdlee} classified the warped product spaces with gradient Ricci solitons. We propose the papers (\cite{bam19, buring22, bbcnazrul24, ppbbcgmn, csvv22, dcy18, gda21, hck23, hpr21, hpv22, hpm22, yha22, ppbbcnasi2, ppdgds, ppbbcnasi1, ppksaip, ppksfacta}) for deep understanding and various interesting results on respective solitons and various examples.\\
\indent Motivated by the above research we study $(\alpha, \beta)- (RYS)_n$ and $(\alpha, \beta)- (GRYS)_n$ on $ (LPS)_n$. The preliminaries discussed in section 2. After that we study $(\alpha, \beta)- (RYS)_n$ in section 3. In section 4 we study $(\alpha, \beta)- (RYS)_n$ in perfect fluid flow. Further in section 5 and 6, we deal with the cosmological models. Section 7 has been explored with $\eta$-parallel Ricci tensor and in section 8, we study Poisson manifold admitting $(\alpha, \beta)- (RYS)_n$. The gradient case $(\alpha, \beta)-(GRYS)_n$ has been studied in section 9. The last section includes an example of $(\alpha, \beta)-(RYS)_n$ admitting $ (LPS)_n$.
\section{Preliminaries}
	Let $M^{n}$ (Riemannian manifold) denotes the inclusion of the triplet data $(\phi, \xi, \eta)$ then we say $(M^{n}, \phi, \xi, \eta)$ is an almost contact metric structure that satisfy\\
\begin{equation}\label{eqn. 2.1}
	\begin{split}
		&\phi\xi=0,\\&\eta(\phi \zeta_{1})=0, \\&\eta(\xi)= -1, \\&\phi^{2}\zeta_{1}= \zeta_{1}+\eta(\zeta_{1})\xi.
	\end{split}
\end{equation}
 Here, $\phi$ denotes a tensor field ($(1,1)$-type) and $\eta$ we call $1$-form on $M^{n}$ defined by $\eta(\zeta_{1})=g(\zeta_{1},\xi)$. Some additional following properties also hold on almost contact metric manifold. 
\begin{equation}\label{eqn. 2.2}
	g(\phi \zeta_{1}, \phi \zeta_{2}) = g(\zeta_{1},\zeta_{2}) - \eta(\zeta_{1})\eta(\zeta_{2})
\end{equation}
equivalently
\begin{equation}\label{eqn. 2.3}
	g(\zeta_{1},\phi \zeta_{2}) = -g(\phi \zeta_{1},\zeta_{2})
\end{equation}
\begin{equation*}
	g(\zeta_{1},\xi) = \eta(\zeta_{1}),
\end{equation*}
for all the vector fields $\zeta_{1},\zeta_{2} \in \chi(M)$.\\
The structure $(M^{n}, \phi, \xi, \eta, g)$ poured to $K$-contact, for being killing to characteristic vector field $\xi$. For Lorentzian para Sasakian $(LPS)_n$ manifold with killing vector field, we have
\begin{equation}\label{eqn. 2.4}
	\nabla _{\zeta_{1}}\xi = \phi \zeta_{1},
\end{equation}
\begin{equation}\label{eqn. 2.5}
	(L_{\xi}g)(\zeta_{1},\zeta_{2}) = 0,
\end{equation}
\begin{equation}\label{eqn. 2.6}
	\phi \xi=0,~~~~~\eta(\phi \zeta_{1})=0,~~~~~ rank ~~\phi = n-1.
\end{equation}
Also 
\begin{equation}\label{eqn. 2.7}
	(\nabla_{\zeta_{1}}\eta)(\zeta_{2}) = \nabla_{\zeta_{1}}\eta(\zeta_{2}) - \eta(\nabla_{\zeta_{1}}\zeta_{2}) = g(\zeta_{2}, \nabla_{\zeta_{1}}\zeta_{2}).
\end{equation}
In view of (\ref{eqn. 2.4}) and (\ref{eqn. 2.7}) we have
\begin{equation}\label{eqn. 2.8}
	(\nabla_{\zeta_{1}}\eta)(\zeta_{2}) = g(\zeta_{2},\phi \zeta_{1}).
\end{equation}
Let structure $(\phi,\xi,\eta, g)$ on $(LPS)_{n}$ $M^{n}$ ($n$ being dimension) then
\begin{equation}\label{eqn. 2.9}
	g(R(\zeta_{1},\zeta_{2})\zeta_{3},\xi)=\eta(R(\zeta_{1},\zeta_{2})\zeta_{3})=g(\zeta_{2},\zeta_{3})\eta(\zeta_{1})-g(\zeta_{1},\zeta_{3})\eta(\zeta_{2}),
\end{equation}
\begin{equation}\label{eqn. 2.10}
	R(\xi, \zeta_{1})\zeta_{2}=g(\zeta_{1},\zeta_{2})\xi-\eta(\zeta_{2})\zeta_{1},
\end{equation}
\begin{equation}\label{eqn. 2.11}
	R(\zeta_{1},\zeta_{2})\xi = \eta(\zeta_{2})\zeta_{1}-\eta(\zeta_{1})\zeta_{2},
\end{equation}
\begin{equation}\label{eqn. 2.12}
	R(\xi, \zeta_{1})\xi=\zeta_{1}+\eta(\zeta_{1})\xi,
\end{equation}
\begin{equation}\label{eqn. 2.13}
	S(\zeta_{1}, \xi) = (n-1)\eta(\zeta_{1}),
\end{equation}
\begin{equation}\label{eqn. 2.14}
	S(\phi \zeta_{1},\phi \zeta_{2}) = S(\zeta_{1},\zeta_{2})+(n-1)\eta(\zeta_{1})\eta(\zeta_{2}),
\end{equation}
for the vector fields $\zeta_{1},\zeta_{2},\zeta_{3}$, where $R$ and $S$ denote the tensor of Riemann and Ricci.\\
\section{$(\alpha, \beta)-(RYS)_n$ on ($M^{n}, \phi, \xi, \eta, g$)}
\begin{lemma}
	Let $M^{n}$ be an Lorentzian para Sasakian $(LPS)_n$ manifold. If the $(\alpha, \beta)-(RYS)_n$ admits an $(LPS)_n$ structure then it is Einstein manifold and its Ricci curvature is given by equation (\ref{eqn. 3.1}).
\end{lemma}
\begin{proof}Taking $V = \xi$ in (\ref{eqn. 1.2}) and using (\ref{eqn. 2.5}), we get
\begin{equation}\label{eqn. 3.1}
	S(\zeta_{1},\zeta_{2}) = \bigg(\frac{\beta r - 2\lambda}{2\alpha}\bigg)g(\zeta_{1},\zeta_{2}).
\end{equation}
Hence proved.
\end{proof}
\begin{lemma}
	Let $M^{n}$ be an $(LPS)_n$ manifold. If the $(\alpha, \beta)-(RYS)_n$ admits an $(LPS)_n$ structure then (\ref{eqn. 3.2}), (\ref{eqn. 3.3}), (\ref{eqn. 3.4}), (\ref{eqn. 3.5}) and (\ref{eqn. 3.6}) hold good.
\end{lemma}
\begin{proof}
From (\ref{eqn. 3.1}), it is very obvious to obtain the following results:
\begin{equation}\label{eqn. 3.2}
	S(\zeta_{1},\xi) = S(\xi, \zeta_{1}) = \bigg(\frac{\beta r - 2\lambda}{2\alpha}\bigg)\eta(\zeta_{1}),
\end{equation}
\begin{equation}\label{eqn. 3.3}
	S(\xi,\xi) = \bigg(\frac{2\lambda-\beta r}{2\alpha}\bigg),
\end{equation}
\begin{equation}\label{eqn. 3.4}
	Q\zeta_{1} = \bigg(\frac{\beta r - 2\lambda}{2\alpha}\bigg)\zeta_{1},
\end{equation}
\begin{equation}\label{eqn. 3.5}
	Q \xi = \bigg(\frac{\beta r - 2\lambda}{2\alpha}\bigg)\xi,
\end{equation}
\begin{equation}\label{eqn. 3.6}
	SCA_{(\alpha, \beta)-RYS}= n\bigg(\frac{\beta r - 2\lambda}{2\alpha}\bigg).
\end{equation}
$SCA_{(\alpha, \beta)-RYS}$ denotes the scalar curvature of $(\alpha, \beta)- (RYS)_n$ admitting $(LPS)_n$.
\end{proof}
\begin{theorem}
	Let $(\alpha, \beta)- (RYS)_n$ admits an $(LPS)_n$ manifold having a non-constant scalar then $\alpha$ and $\beta$ are related by $\alpha=\frac{n\beta}{2}$.
\end{theorem}
\begin{proof} The derivative (covariant) of (\ref{eqn. 3.4}) yields 
\begin{equation}\label{eqn. 3.9}
	(\nabla_{K}Q\zeta_{1}) = \frac{\beta}{2\alpha}(Kr)\zeta_{1}.
\end{equation}
Applying the metric both side, above equation reduces to
\begin{equation}\label{eqn. 3.10}
	g((\nabla_{K}Q)(\zeta_{1},\zeta_{2}) = \frac{\beta}{2\alpha}(Kr)g(\zeta_{1},\zeta_{2}).
\end{equation}
Arranging $\zeta_{1}=\zeta_{2}=e_{i}$ in (\ref{eqn. 3.10}) and contracting, we get
\begin{equation}\label{eqn. 3.11}
	g((\nabla_{K}Q)e_{i},e_{i}) = n\frac{\beta}{2\alpha}(Kr),
\end{equation}
which implies
\begin{equation}\label{eqn. 3.13}
	(Kr)\big(\frac{n\beta}{2\alpha}-1\big)=0.
\end{equation}
For being non-constant scalar curvature $r$, we have the result.
\end{proof}
\begin{theorem}
	Let $(\alpha, \beta)- (RYS)_n$ admits an $(LPS)_n$ manifold having a non-constant scalar curvature then soliton function is given by
	\begin{equation*}\label{eqn. 3.18}
		\lambda = \frac{\beta}{2}(r-n(n-1)).
	\end{equation*}
\end{theorem}
\begin{proof}
Using (\ref{eqn. 2.5}) in (\ref{eqn. 1.2}),  we get 
\begin{equation}\label{eqn. 3.14}
2\alpha S(\zeta_{1},\zeta_{2}) = (\beta r-2\lambda)g(\zeta_{1},\zeta_{2}).	
\end{equation}
Substituting $\zeta_{2}=\xi$ in (\ref{eqn. 3.14}) and then applying (\ref{eqn. 2.13}), result obtain
\begin{equation}\label{eqn. 3.16}
	2\alpha (n-1)\eta(\zeta_{1}) = (\beta r-2\lambda)\eta(\zeta_{1}).
\end{equation}
Since, $\eta(\zeta_{1})$ is non-vanishing, the above equation produces
\begin{equation}\label{eqn. 3.17}
	\lambda = \frac{\beta r}{2} - (n-1)\alpha.
\end{equation}
With the help of theorem (3.1), equation (\ref{eqn. 3.17}) yields
\begin{equation}\label{eqn. 3.18}
		\lambda = \frac{\beta}{2}(r-n(n-1)),
\end{equation}
which completes the proof.
\end{proof}
\section{$(\alpha, \beta)- (RYS)_4$ Admitting $(LPS)_4$ Space-time}
In this section, we consider the $(LPS)_4$ spacetime to study. The Einstein's field equation {\it (briefly: EFE)} defines the following relation
\begin{equation}\label{eqn. 4.1}
	S(\zeta_{1},\zeta_{2}) - \frac{r}{2}g(\zeta_{1},\zeta_{2}) + \mu g(\zeta_{1},\zeta_{2}) = \tau T(\zeta_{1},\zeta_{2}).
\end{equation}
 Here, $\mu$ denotes the cosmological term and $\zeta_{1},\zeta_{2}$, the vector fields. $\tau$ being the gravitational constant and T deals for $(0,2)$ type energy momentum tensor (energy tensor).\\
 The below equation reveals the information regarding perfect fluid and govern by
\begin{equation}\label{eqn. 4.2}
	T(\zeta_{1},\zeta_{2}) = (\rho - p)A(\zeta_{1})A(\zeta_{2}) + pg(\zeta_{1},\zeta_{2}).
\end{equation}
Here,  energy tensor and energy density function denote by T, whereas $p$ indicates for fluid isotropic pressure. Also, we consider here a non-zero $1$-form by $A$ that define as $g(\zeta_{1},V) = A(\zeta_{1})$; V being the flow vector field.\\

In our discussion, we consider $\xi$ as the flow vector field and (\ref{eqn. 4.2}) changes in a new form as
\begin{equation}\label{eqn. 4.3}
	T(\zeta_{1},\zeta_{2}) = (\rho - p)\eta(\zeta_{1})\eta(\zeta_{2}) + pg(\zeta_{1},\zeta_{2}).	
\end{equation}
Substituting the value of Ricci tensor from (\ref{eqn. 3.1}) in (\ref{eqn. 4.1}), we have
\begin{equation}\label{eqn. 4.4}
	T(\zeta_{1},\zeta_{2}) = \frac{1}{\tau}\bigg[\mu-(\frac{\beta r - 2\lambda}{2\alpha})\bigg]g(\zeta_{1},\zeta_{2}).
\end{equation}
Using theorem (3.1), the energy tensor takes the form
\begin{equation}\label{eqn. 4.4}
	T(\zeta_{1},\zeta_{2}) = \frac{1}{4\tau}\bigg[(4\mu-r+\frac{2\lambda}{\beta}\bigg]g(\zeta_{1},\zeta_{2}).
\end{equation}
The conclusion is below:
\begin{theorem}
Let $(\alpha, \beta)- (RYS)_4$ admits an $(LPS)_4$ space-time with non-constant scalar curvature then energy tensor is an Einstein like tensor and given by
\begin{equation*}\label{eqn. 4.5}
	T(\zeta_{1},\zeta_{2}) = \frac{1}{4\tau}\bigg[(4\mu-r+\frac{2\lambda}{\beta}\bigg]g(\zeta_{1},\zeta_{2}).
\end{equation*}
\end{theorem}
\begin{theorem}
	Let $(\alpha, \beta)- (RYS)_4$ admits an $(LPS)_4$ space-time with non-constant scalar curvature then soliton function $\lambda$ is given by 
	\begin{equation*}\label{eqn. 4.10}
		\lambda= {\frac{\beta}{2}\{4\tau(2p-\rho)+r-4\mu\}}.
	\end{equation*}
\end{theorem}
\begin{proof} Equating (\ref{eqn. 4.3}) and (\ref{eqn. 4.4}), we obtain
\begin{equation}\label{eqn. 4.8}
	\frac{1}{4\tau}\bigg[(4\mu-r+\frac{2\lambda}{\beta}\bigg]g(\zeta_{1},\zeta_{2})= (\rho - p)\eta(\zeta_{1})\eta(\zeta_{2})+p g(\zeta_{1},\zeta_{2}).
\end{equation}
Arranging $\zeta_{1}=\zeta_{2}=\xi$, the above equation (\ref{eqn. 4.8}) yields the soliton function
\begin{equation}\label{eqn. 4.9}
\lambda= {\frac{\beta}{2}\{4\tau(2p-\rho)+r-4\mu\}}.
\end{equation}
Hence the proof has been over.
\end{proof}
\begin{theorem}
	Let $(\alpha, \beta)$ steady-$(RYS)_n$ admits an $(LPS)_n$ space-time then non-constant scalar curvature is four times of cosmoligical scalar function $\mu$, that is:
		\begin{equation*}\label{eqn. 4.12}
		r=4\mu.
	\end{equation*}
\end{theorem}
\begin{proof} For the steady soliton we know that $\lambda$ vanishes. Therefore, using (\ref{eqn. 4.10}) we can calculate the value of cosmological constant $\mu$ as
\begin{equation}\label{eqn. 4.11}
	\mu = \frac{\beta r-2\lambda}{2\alpha}.
\end{equation}
Applying theorem (3.1), completes the proof.
\end{proof}
\section{Cosmological Model on $(\alpha, \beta)- (RYS)_4$ Admitting $(LPS)_4$ Space-time}
This section, contains the results on $(LPS)_4$ space-time with considering a killing vector field. Here, we consider the (EFE), that does not consist cosmological term, given by
\begin{equation}\label{eqn. 5.1}
		S(\zeta_{1},\zeta_{2}) - \frac{r}{2}g(\zeta_{1},\zeta_{2}) = \tau T(\zeta_{1},\zeta_{2}).
\end{equation}
Using equation (\ref{eqn. 3.1}), (\ref{eqn. 4.3}) and (\ref{eqn. 5.1}), we obtain
\begin{equation}\label{eqn. 5.2}
	\bigg[\frac{r(\beta-\alpha)-2\lambda}{2\alpha}-\tau p \bigg]g(\zeta_{1},\zeta_{2}) = \kappa(\rho+p)A(\zeta_{1})A(\zeta_{2}).
\end{equation}
Contracting $\zeta_{1}$ and $\zeta_{2}$ in (\ref{eqn. 5.2}), we get
\begin{equation}\label{eqn. 5.3}
	r = \frac{\alpha \tau(-\rho+3p)+2\lambda}{2(\beta-\alpha)}.
\end{equation}
Substituting (\ref{eqn. 5.3}) in (\ref{eqn. 3.1}), we get
\begin{equation}\label{eqn. 5.4}
	S(\zeta_{1},\zeta_{2})= \bigg[\frac{\beta\alpha\tau(-\rho+3p)-2\lambda(\beta-2\alpha)}{4\alpha(\beta-\alpha)}\bigg]g(\zeta_{1},\zeta_{2}).
\end{equation}
As we know the Ricci operator $Q$ is defined as
\begin{equation}\label{eqn. 5.5}
	g(Q\zeta_{1},\zeta_{2})=S(\zeta_{1},\zeta_{2})~~ and ~~ S(Q\zeta_{1},\zeta_{2})=S^{2}(\zeta_{1},\zeta_{2}).
\end{equation}
Thus, we have
\begin{equation}\label{eqn. 5.6}
	A(Q\zeta_{1}) = g(Q\zeta_{1},\xi)=S(\zeta_{1},\xi).
\end{equation}
 Then from equation (\ref{eqn. 5.4}) and (\ref{eqn. 5.5}), we have
 \begin{equation}\label{eqn. 5.7}
 	S(Q\zeta_{1},\zeta_{2}) = S^{2}(\zeta_{1},\zeta_{2}) = \bigg[\frac{\beta\alpha\tau(-\rho+3p)-2\lambda(\beta-2\alpha)}{4\alpha(\beta-\alpha)}\bigg]^{2}g(\zeta_{1},\zeta_{2}).
 \end{equation}
 Contraction over $\zeta_{1}$ and $\zeta_{2}$ in (\ref{eqn. 5.7}), we  have
 \begin{equation}\label{eqn. 5.8}
 	\lVert Q \rVert^{2} = \bigg[\frac{\beta\alpha\tau(-\rho+3p)-2\lambda(\beta-2\alpha)}{2\alpha(\beta-\alpha)}\bigg]^{2}.
 \end{equation}
Using theorem (3.1), we get
\begin{equation}\label{eqn. 5.8}
	\lVert Q \rVert^{2} =\frac{1}{16} \bigg[\tau(\rho-3p)-\frac{3\lambda}{\beta}\bigg]^{2}.
\end{equation}
\begin{theorem}
	If the perfect fluid space-time on $(\alpha, \beta)- (RYS)_4$ admitting $(LPS)_4$ satisfy the {\it (EFE)} not consisting the cosmological term, the length of Ricci operator is given by
	 \begin{equation*}
		\lVert Q \rVert^{2} =\frac{1}{16} \bigg[\tau(\rho-3p)-\frac{3\lambda}{\beta}\bigg]^{2}.
	\end{equation*}
\end{theorem}
Now, if we put $\rho = 3p$ (the condition of perfect fluid to be radiation fluid) in equation (\ref{eqn. 5.8}), then we have
\begin{theorem}
	If the perfect fluid is radiation fluid (i.e $\rho=3p$) in $(\alpha, \beta)- (RYS)_4$ admitting $(LPS)_4$, then length of Ricci operator is given by
	\begin{equation*}
		\lVert Q \rVert^{2} = \bigg[\frac{3\lambda}{4\beta}\bigg]^{2}.
	\end{equation*}
\end{theorem}
\section{Dust Fluid on $(\alpha, \beta)- (RYS)_4$ Admitting $(LPS)_4$ Space-time}
The dust fluid and the eneregy tensor can be modelled by
\begin{equation}\label{eqn. 6.1}
	T(\zeta_{1},\zeta_{2})= \rho A(\zeta_{1})A(\zeta_{2}).
\end{equation}
Using equation (\ref{eqn. 5.1}) and (\ref{eqn. 6.1}), the expression becomes 
\begin{equation}\label{eqn. 6.2}
		S(\zeta_{1},\zeta_{2}) - \frac{r}{2}g(\zeta_{1},\zeta_{2}) = \tau \rho A(\zeta_{1})A(\zeta_{2}).
\end{equation}
On contracting over $\zeta_{1}$ and $\zeta_{2}$ in above equation, we get
\begin{equation}\label{eqn. 6.3}
	r= \tau \rho.
\end{equation}
By using equation (\ref{eqn. 6.3}) in equation (\ref{eqn. 3.1}), the Ricci tensor of perfect fluid $(\alpha, \beta)- (RYS)_4$ admitting $(LPS)_4$ spacetime becomes
\begin{equation}\label{eqn. 6.4}
	S(\zeta_{1},\zeta_{2}) = \bigg[\frac{\beta \tau\rho -2\lambda}{2\alpha} \bigg]g(\zeta_{1},\zeta_{2}).
\end{equation}
Thus, by virtue of equation  (\ref{eqn. 5.7}) and  (\ref{eqn. 6.4}), we obtain
\begin{equation}\label{eqn. 6.5}
	S(Q\zeta_{1},\zeta_{2}) = \bigg[\frac{\beta \tau\rho -2\lambda}{2\alpha} \bigg]^{2}g(\zeta_{1},\zeta_{2}).
\end{equation}
Contracting $\zeta_{1}$ and $\zeta_{2}$ in above equation, we have
\begin{equation}\label{eqn. 6.6}
	\lVert Q \rVert^{2} = \bigg[\frac{\beta \tau\rho -2\lambda}{2\alpha} \bigg]^{2}.
\end{equation}
Using theorem (3.1), we have
\begin{equation}\label{eqn. 6.6a}
	\lVert Q \rVert^{2} = \frac{1}{16}\bigg[\tau\rho -\frac{2\lambda}{\beta} \bigg]^{2}.
\end{equation}
\begin{theorem}
		If the dust cosmological model on $(\alpha, \beta)- (RYS)_4$ admitting $(LPS)_4$ space-time satisfy the {\it (EFE)} not consisting the cosmological term, the length of Ricci operator is given as:
	\begin{equation*}
		\lVert Q \rVert^{2} = \frac{1}{16}\bigg[\tau\rho -\frac{2\lambda}{\beta} \bigg]^{2}.
	\end{equation*}	
\end{theorem}
Now, using equation (\ref{eqn. 3.1}) and (\ref{eqn. 6.2}), we get
\begin{equation}\label{eqn. 6.7}
	\bigg[\frac{\beta r-2\lambda -\alpha r}{2\alpha}\bigg]g(\zeta_{1},\zeta_{2}) = \tau \rho A(\zeta_{1})A(\zeta_{2}).
\end{equation}
Contracting over $\zeta_{1}$ and $\zeta_{2}$ in equation (\ref{eqn. 6.7}), we get
\begin{equation}\label{eqn. 6.8}
	r= \frac{-\alpha \rho \tau + 4\lambda}{2(\beta-\alpha)}.
\end{equation}
By multiplying $A(\zeta_{3})$ in equation (\ref{eqn. 6.7}) and then taking contraction over $\zeta_{2}$ and $\zeta_{3}$, we have
\begin{equation}\label{eqn. 6.9}
	r= \frac{2(\alpha \rho \tau+\lambda)}{\beta-\alpha}.
\end{equation}
From equation (\ref{eqn. 6.8})  and (\ref{eqn. 6.9}), we obtain 
\begin{equation}\label{eqn. 6.10}
	\alpha \rho \tau = 0.
\end{equation}
 As we know $\tau$ is gravitational constant, therefore $\tau \neq0$.
 Thus, we have either $\alpha = 0$ or $\rho=0$.\\
 If $\rho=0$, then from equation (\ref{eqn. 6.1})  and (\ref{eqn. 6.10}), we obtain
 \begin{equation}
 	T(\zeta_{1},\zeta_{2}) = 0.
 \end{equation}
 Therefore, we can state the following theorem:
 \begin{theorem}
 	If the dust cosmological model on $(\alpha, \beta)- (RYS)_4$ admitting $(LPS)_4$ space-time satisfy the {\it (EFE)} not consisting the cosmological term the it is vacuum.
 \end{theorem}
\section{$\eta-$PRT on $(\alpha, \beta)- (RYS)_n$ Admitting $(LPS)_n$}
This section examines $\eta$-parallel Ricci tensor {\it (briefly: $\eta-$PRT)} on $(\alpha, \beta)- (RYS)_n$ admitting $(LPS)_n$. If an $(LPS)_n$ consist {\it $\eta-$PRT} then
\begin{equation}\label{eqn. 8.1}
	g((\nabla_{K}Q)\zeta_{1},\zeta_{2})=0.
\end{equation}
for all smooth vector fields $\zeta_{1},\zeta_{2}$.\\
The following is the very general expansion of Ricci operator
\begin{equation}\label{eqn. 8.2}
	(\nabla_{K}Q)\zeta_{1} = \nabla_{K}Q\zeta_{1} - Q(\nabla_{K}\zeta_{1}).
\end{equation}
Using (\ref{eqn. 3.4}) in (\ref{eqn. 8.2}), yields
\begin{equation}\label{eqn. 8.3}
(\nabla_{K}Q)\zeta_{1} = \frac{\beta}{2\alpha}(Kr)\zeta_{1}.
\end{equation}
Using (\ref{eqn. 8.3}) in (\ref{eqn. 8.1}), we obtain
\begin{equation}\label{eqn. 8.4}
	g((\nabla_{K}Q)\zeta_{1},\zeta_{2})=\frac{\beta}{2\alpha}(Kr)g(\zeta_{1},\zeta_{2}).
\end{equation}
In view of (\ref{eqn. 3.13}), we have $Kr=0$ only if $\big(\frac{n\beta}{2\alpha}-1\big)\neq 0$.
\begin{theorem}
	There exists an $(\alpha, \beta)- (RYS)_n$ admitting $(LPS)_n$ with an $\eta$-parallel Ricci tensor only if $(\frac{n\beta}{2\alpha}-1)\neq0$.
\end{theorem}
\section{Poisson $(\alpha, \beta)- (RYS)_n$ Admitting $(LPS)_n$ Manifold}
This section provides a definition of the statistical Poisson manifold and highlights several significant findings related to the statistical structure on the manifold.
\begin{definition}
	The Poisson manifold consist the codazzi structure and explored by
	\begin{equation}\label{eqn. 9.1}
		(\nabla_{\zeta_{3}}S)(\zeta_{1},\zeta_{2}) = (\nabla_{\zeta_{2}}S)(\zeta_{3},\zeta_{1}).
	\end{equation}
\end{definition}
Taking covariant derivative of (\ref{eqn. 3.1}) and using (\ref{eqn. 2.4}), we get
\begin{equation}\label{eqn. 9.3}
	(\nabla_{\zeta_{3}}S)(\zeta_{1},\zeta_{2})= \frac{\beta}{2\alpha}(\zeta_{3}r)g(\zeta_{1},\zeta_{2}).
\end{equation}
Also, we can calculate,
\begin{equation}\label{eqn. 9.4}
	(\nabla_{\zeta_{2}}S)(\zeta_{3},\zeta_{1})= \frac{\beta}{2\alpha}(\zeta_{2}r)g(\zeta_{3},\zeta_{1}).
\end{equation}
In view of (\ref{eqn. 9.1}), (\ref{eqn. 9.3}) and (\ref{eqn. 9.4}), we have
\begin{equation}\label{eqn. 9.5}
	\frac{\beta}{2\alpha}(\zeta_{3}r)g(\zeta_{1},\zeta_{2})= \frac{\beta}{2\alpha}(\zeta_{2}r)g(\zeta_{3},\zeta_{1}).
\end{equation}
Put $\zeta_{1}=\xi$ in (\ref{eqn. 9.5}), we obtain
\begin{equation}\label{eqn. 9.6}
	\frac{\beta}{2\alpha}[(\zeta_{3}r)\eta(\zeta_{2})-(\zeta_{2}r)\eta(\zeta_{3})]=0.
\end{equation}
Put $\zeta_{3}=\xi$ in (\ref{eqn. 9.6}) and using lemma (3.1), we have
\begin{equation}
	\frac{\beta}{2\alpha}(\zeta_{2}r)=0.
	\end{equation}
So, we have either $\frac{\beta}{2\alpha}=0$ or $(\zeta_{2}r)=0$.\\
As defined earlier $\alpha$ and $\beta$ are non-vanishing constants, therefore $\zeta_{2}r=0$ which is possible only if $\big(\frac{n\beta}{2\alpha}-1\big)\neq 0$.
Thus, we have
\begin{theorem}
	An $(\alpha, \beta)- (RYS)_n$ admitting $(LPS)_n$ is the statistical Poisson manifold only if $\big(\frac{n\beta}{2\alpha}-1\big)\neq 0$ and the scalar curvature is constant.
\end{theorem}
\section{Gradient $(\alpha, \beta)- (RYS)_n$ Admitting $(LPS)_n$ Manifold}
Suppose that $(\alpha, \beta)- (GRYS)_n$ admits $(LPS)_n$ manifold then (\ref{eqn. 1.3}) turns to
\begin{equation}\label{eqn. 10.1}
	\nabla_{\zeta_{1}}Df = (\lambda-\frac{1}{2}\beta r)\zeta_{1}-\alpha Q\zeta_{1}.
\end{equation}
The derivative (covariantly) of (\ref{eqn. 10.1}) along $\zeta_{2}$, produces
\begin{equation}\label{eqn. 10.2}
	\nabla_{\zeta_{2}}\nabla_{\zeta_{1}}Df = (\lambda-\frac{1}{2}\beta r)\nabla_{\zeta_{2}}\zeta_{1}- \frac{\beta}{2}(\zeta_{2}r)\zeta_{1}-\alpha \nabla_{\zeta_{2}}Q\zeta_{1}.
\end{equation}
Interchanging $\zeta_{1}$ and $\zeta_{2}$ in (\ref{eqn. 10.2}) refer
\begin{equation}\label{eqn. 10.3}
	\nabla_{\zeta_{1}}\nabla_{\zeta_{2}}Df = (\lambda-\frac{1}{2}\beta r)\nabla_{\zeta_{1}}\zeta_{2}- \frac{\beta}{2}(\zeta_{1}r)\zeta_{2}-\alpha \nabla_{\zeta_{1}}Q\zeta_{2}.
\end{equation}
It is quite simple to arrange the above information to form the following
\begin{equation}\label{eqn. 10.4}
R(\zeta_{1},\zeta_{2})Df= \frac{\beta}{2}[(\zeta_{2}r)\zeta_{1}-(\zeta_{1}r)\zeta_{2}]-\alpha[(\nabla_{\zeta_{1}}Q)\zeta_{2}-(\nabla_{\zeta_{2}}Q)\zeta_{1}].
\end{equation}
On the other hand, equation (3.4) results as
\begin{equation}\label{eqn. 10.5}
	(\nabla_{\zeta_{1}}Q)\zeta_{2}-(\nabla_{\zeta_{2}}Q)\zeta_{1} = \frac{\beta}{2\alpha}[(\zeta_{1}r)\zeta_{2}-(\zeta_{2}r)\zeta_{1}].
\end{equation}
Using (\ref{eqn. 10.5}) in (\ref{eqn. 10.4}), we get
\begin{equation}\label{eqn. 10.6}
	R(\zeta_{1},\zeta_{2})Df = \frac{\beta}{2}[(\zeta_{2}r)\zeta_{1}-(\zeta_{1}r)\zeta_{2}]- \alpha [\frac{\beta}{2\alpha}[(\zeta_{1}r)\zeta_{2}-(\zeta_{2}r)\zeta_{1}]],
\end{equation}
which can be simplified as
\begin{equation}\label{eqn. 10.7}
		R(\zeta_{1},\zeta_{2})Df = \beta[(\zeta_{2}r)\zeta_{1}-(\zeta_{1}r)\zeta_{2}].
\end{equation}
Contracting (\ref{eqn. 10.7}) and using lemma (3.1) we get
\begin{equation}\label{eqn. 10.8}
	S(\zeta_{2},Df)= n\beta(\zeta_{2}r).
\end{equation}
Alteration of $\zeta_{1}$ by $Df$ in (\ref{eqn. 3.1}), yields
\begin{equation}\label{eqn. 10.9}
		S(\zeta_{2},Df)= \bigg(\frac{\beta r - 2\lambda}{2\alpha}\bigg)(\zeta_{2}f). 
\end{equation}
In view of (\ref{eqn. 10.8}) and (\ref{eqn. 10.9}), we have
\begin{equation}\label{eqn. 10.10}
	n\beta(\zeta_{2}r) = \bigg(\frac{\beta r - 2\lambda}{2\alpha}\bigg)(\zeta_{2}f).
\end{equation}
Since $\zeta_{2}r =0$ is possible only if $\bigg(\frac{n\beta}{2\alpha}-1\bigg) \neq 0$. Hence, we have
\begin{equation}\label{eqn. 10.12}
	\bigg(\frac{\beta r - 2\lambda}{2\alpha}\bigg)(\xi f) = 0.
\end{equation}
 As $(\xi f) \neq 0$, the soliton function reduce to
\begin{equation}\label{eqn. 10.12a}
	\lambda=\frac{\beta r}{2}.
\end{equation}
\begin{theorem}
In a $(\alpha, \beta)- (GRYS)_n$ admitting $(LPS)_n$ manifold the soliton function is given by.
\begin{equation*}\label{eqn. 10.12a}
	\lambda=\frac{\beta r}{2}.
\end{equation*}
provided $\bigg(\frac{n\beta}{2\alpha}-1\bigg) \neq 0$.
\end{theorem}
The metric inner product on (\ref{eqn. 10.8}) with $\xi$ results
\begin{equation}\label{eqn. 10.14}
\eta(\zeta_{2})(\zeta_{1}f)-\eta(\zeta_{1})(\zeta_{2}f)= \frac{\beta}{2}[(\zeta_{2}r)\eta(\zeta_{1})-(\zeta_{1}r)\eta(\zeta_{2})].	
\end{equation}
Put $\zeta_{1}=\xi$ in (\ref{eqn. 10.14}) and using lemma (3.1), we have
\begin{equation}\label{eqn. 10.15}
	\eta(\zeta_{2})(\xi f)-\eta(\xi)(\zeta_{2}f)= \frac{\beta}{2}[(\zeta_{2}r)\eta(\xi)-(\xi r)\eta(\zeta_{2})].	
\end{equation}
Simplifying, above equation reduces to
\begin{equation}\label{eqn. 10.16}
	(\zeta_{2}r)= -\frac{2}{\beta}(\zeta_{2}f).
\end{equation}
Using (\ref{eqn. 10.16}) in (\ref{eqn. 10.10}) and then applying theorem (3.1), we receive
\begin{equation}\label{eqn. 10.17}
	\lambda=\frac{\beta(r+2n^2)}{2}.
\end{equation}
Thus, we can state
\begin{theorem}
	The soliton function of $(\alpha, \beta)-(GRYS)_n$ admitting $(LPS)_n$ Manifold, having non-constant scalar curvature, is given by
	\begin{equation*}\label{eqn. 10.17}
		\lambda=\frac{\beta(r+2n^2)}{2}.
	\end{equation*}
\end{theorem}

\section{Example}
\begin{example}
Let's have a look at a differentiable manifold in four dimensions, $M^4 = (\zeta_{1}, \zeta_{2}, \zeta_{3}, \zeta_{4})\in \Re : (\zeta_{1}, \zeta_{2}, \zeta_{3}, \zeta_{4}) \neq 0$, where $(\zeta_{1}, \zeta_{2}, \zeta_{3}, \zeta_{4})$ is the standard coordinate in four-dimensional real space. Considering that each point on M has a collection of linearly independent vector fields, let's call them $(f_{1},f_{2},f_{3},f_{4})$, and is defined by 
\begin{equation}\label{}
	f_{1}=f^{\zeta_{1}-a\zeta_{4}}\frac{\partial}{\partial \zeta_{1}},~~~~ 	f_{2}=f^{\zeta_{2}-a\zeta_{4}}\frac{\partial}{\partial \zeta_{2}},~~~~ 	f_{3}=f^{\zeta_{3}-a\zeta_{4}}\frac{\partial}{\partial \zeta_{3}},~~~~	f_{4}=\frac{\partial}{\partial \zeta_{4}}
\end{equation} 
where a is non-zero constant and $\{\frac{\partial}{\partial \zeta_{1}}, \frac{\partial}{\partial \zeta_{2}}, \frac{\partial}{\partial \zeta_{3}}, \frac{\partial}{\partial \zeta_{4}}\}$ denotes the standard basis of $M^4$.\\
 Let the metric (Lorentzian) and the $1$-form $\eta$ components on $M^4$ are concreted by
 \begin{equation}
g_{cd} = g(f_{c},f_{d})	=
	\begin{cases}
		 -1, & \text{if}\  c=d=4\\
		0, & \text{if}\ c\neq d\\
	     1, & \text{if}\ c=d\neq 4
	\end{cases}
\end{equation}
and
\begin{equation}
	\eta(\zeta_{1})=g(\zeta_{1},f_{4}).
\end{equation}
for any $\zeta_{1} \in \Gamma(TM)$ on $M^4$. 
\end{example}
If
 $\phi(f_{1})=f_{1}, \phi(f_{2})=f_{2}, \phi(f_{3})=f_{3}, \phi(f_{4})=0$,\\
  are the tensor field then the following relationships may be easily verified by g's and $\phi$'s linearity properties:
\begin{equation}
	\eta(f_{4})=-1,~~~~\phi^{2}\zeta_{1}= \zeta_{1}+\eta(X)f_{4},~~~~	g(\phi \zeta_{1}, \phi \zeta_{2}) = g(\zeta_{1},\zeta_{2}) - \eta(\zeta_{1})\eta(\zeta_{2}).
\end{equation}
The Lie bracket's (non-vanishing) components are determined as follows:
\begin{equation}
	[f_{1},f_{4}]=af_{1},~~~~[f_{2},f_{4}]=af_{2}~~~~[f_{3},f_{4}]=af_{3}.
\end{equation}
The Koszul's formula yields for $f_{4}=\xi$
\begin{equation}
	\begin{split}
		&\nabla_{f_{1}}f_{1}=af_{4},~~\nabla_{f_{1}}f_{2}=0,~~\nabla_{f_{1}}f_{3}=0,~~\nabla_{f_{1}}f_{4}=af_{1}, \\&\nabla_{f_{2}}f_{1}=0,~~\nabla_{f_{2}}f_{2}=af_{4},~~\nabla_{f_{2}}f_{3}=0,~~\nabla_{f_{2}}f_{4}=af_{2},\\&\nabla_{f_{3}}f_{1}=0,~~\nabla_{f_{3}}f_{2}=0,~~\nabla_{f_{3}}f_{3}=af_{4},~~\nabla_{f_{3}}f_{4}=af_{3},\\&\nabla_{f_{4}}f_{1}=0,~~\nabla_{f_{4}}f_{2}=0,~~\nabla_{f_{4}}f_{3}=0,~~\nabla_{f_{4}}f_{4}=0.
	\end{split}
\end{equation}
The curvature components R, Ricci components S and the scalar tensor r are obtained respectively
\begin{equation}
	\begin{split}
		&R(f_{1},f_{2})f_{1}=-a^{2}f_{2},~~~R(f_{1},f_{3})f_{1}=-a^{2}f_{3},~~~R(f_{1},f_{4})f_{1}=-a^{2}f_{4},\\&R(f_{1},f_{2})f_{2}=a^{2}f_{1},~~~R(f_{2},f_{3})f_{2}=-a^{2}f_{3},~~~R(f_{2},f_{4})f_{2}=-a^{2}f_{4},\\&R(f_{1},f_{3})f_{3}=a^{2}f_{1},~~~R(f_{2},f_{3})f_{3}=a^{2}f_{2},~~~R(f_{3},f_{4})f_{3}=-a^{2}f_{4},\\&R(f_{1},f_{4})f_{4}=-a^{2}f_{1},~~~R(f_{2},f_{4})f_{4}=-a^{2}f_{2},~~~R(f_{3},f_{4})f_{4}=-a^{2}f_{3}.
	\end{split}
\end{equation}
\begin{equation}\label{10.77}
	S(f_{1},f_{1})=	S(f_{2},f_{2})=	S(f_{3},f_{3})= 3a^{2},~~~~	S(f_{4},f_{4})= -3a^{2},
\end{equation}
\begin{equation}\label{10.79}
	r= 6a^{2}.
\end{equation}
If we consider $\lambda=3a^{2}(\beta-\alpha)$ and the values from the equation (10.8) and (10.9), equation (3.1) satifies.\\ Hence, it is an example of four-dimensional $(\alpha,\beta)-(RYS)_4$ admitting $(LPS)_4$.\\

Using theorem (3.1) in $\lambda=3a^{2}(\beta-\alpha)$, we get
\begin{equation}\label{10.78}
	 \lambda=-3\beta a^{2}.
\end{equation}
Equating (\ref{10.78}) and (\ref{eqn. 3.18}), yield
 \begin{equation}\label{10.80}
 	a=\pm 1.
 \end{equation}
Hence, for $a=\pm 1$, the defined metric in (10.1) and the values obtained in (10.8) and (10.9) satisfy the theorem (3.1) and theorem (3.2).
\section*{Conclusion}
The Lorentzian metric and its generalizations play very important role in the study of cosmology. The perfect fluid, that depends on the lorentzian metric, is an example to represents the universe. Here, we study the soliton metric on Lorentzian manifold and discuss some fruitful investigation. Specially, the cosmological models have been discussed and the length of Ricci operator has been obtained and proved that spacetime is vaccum. The Poisson structure is an interesting topic to investigate in cosmology and mathematical physics. The soliton function has been discussed under the two different conditions, whether the scalar curvature is constant or not.
\section*{Acknowledgement}
There does not any financial support provide to conduct this research.\\
All the authors made equal contribution to complete the manuscript.

\end{document}